\documentclass[10pt]{article}
\usepackage[a4paper,margin=25mm]{geometry}
\usepackage{amsmath,amssymb,amsthm}
\usepackage{microtype}
\usepackage[hidelinks]{hyperref}
\hypersetup{
  pdftitle={An Explicit Bound for the First Dirichlet Eigenvalue via Two Moments},
  pdfauthor={Thomas Schuermann}}
\allowdisplaybreaks[1]
\newtheorem{theorem}{Theorem}[section]
\newtheorem{lemma}[theorem]{Lemma}
\newtheorem{proposition}[theorem]{Proposition}
\newtheorem{corollary}[theorem]{Corollary}

\newcommand{\Ric}{\operatorname{Ric}}
\newcommand{\InRad}{\operatorname{InRad}}
\newcommand{\E}{\mathbb E}

\title{An Explicit Bound for the First Dirichlet Eigenvalue\\
via Two Moments}
\author{Thomas Sch\"urmann\\[0.2em]
D\"usseldorf, Germany
}
\date{}

\begin{document}
\maketitle

\begin{abstract}
Let $(M^n,g)$, $n\ge2$, be a compact connected Riemannian manifold with
smooth nonempty boundary, $\Ric_M\ge(n-1)Kg$ for $K>0$, and nonnegative
outward boundary mean curvature. We derive an explicit algebraic
lower bound $B_\eta$ for the first Dirichlet eigenvalue from Kasue's
mixed inradius model. With $R$ the inradius and $s=4KR^2/\pi^2$,
the bound involves one square root and approximates the model
eigenvalue with relative error $O_n(s^4)$ as $s\downarrow0$ for fixed $n$.
The construction uses the first two moments of a positive
representation of the averaged model potential and a sharp
pointwise corrector estimate. For the chosen spectral transfer, the
first two moments determine an optimal bound.
We also show that the model eigenvalue is the geometric infimum at
fixed $(n,K,R)$. For $0<s<1$, it is approached by smooth
collapsing balls with sectional curvature at least $K$ and totally
geodesic boundary. At the maximal inradius, the model value is
attained by the round hemisphere.
The bound strictly improves the variance refinement throughout
$0<s\le1$. We also give a direct algebraic comparison with an
optimized power-test benchmark.
\end{abstract}



\section{Introduction and main results}
\label{sec:intro}

Kasue's comparison bounds the first Dirichlet eigenvalue from below
by the first eigenvalue of a one-dimensional mixed problem.
We construct an explicit algebraic lower bound for this model
eigenvalue and quantify its accuracy. We also show that the model
eigenvalue is the geometric infimum at fixed dimension, curvature
lower bound, and inradius.

Let $(M^n,g)$ be compact and connected, $n\ge2$, with smooth nonempty
boundary. Assume
\begin{equation}
\Ric_M\ge(n-1)Kg,\qquad K>0,\qquad H_{\partial M}\ge0.
\label{eq:assumptions}
\end{equation}
Here $H_{\partial M}=(n-1)^{-1}\sum_{i=1}^{n-1}
g(\nabla_{e_i}\nu,e_i)$ is the normalized mean curvature, with $\nu$ the
outward unit normal and $(e_i)$ an orthonormal boundary frame.
Denote by $\lambda_1^D(M)$ the first Dirichlet eigenvalue and set
\begin{equation}
R=\InRad(M),\qquad
\Lambda_R=\frac{\pi^2}{4R^2},\qquad
s=\frac{K}{\Lambda_R}=\frac{4KR^2}{\pi^2},\qquad
c=\frac{n-1}{2},
\label{eq:parameters}
\end{equation}
with
\begin{equation}
\alpha=cK,\qquad \beta=c(c+1),\qquad L=\Lambda_R+\alpha.
\label{eq:abL}
\end{equation}
Kasue's inradius comparison gives $R\le\pi/(2\sqrt K)$, hence
$0<s\le1$ \cite[Section~1.2]{Kasue}. 
Define the universal constants
\begin{equation}
a_0=\frac{\pi^2-6}{12},\qquad
a_1=\frac{\pi^4-20\pi^2+120}{120},\qquad
\eta=\frac{a_1}{a_0},
\label{eq:constants}
\end{equation}
and
\begin{equation}
Q_\eta=\frac{\Lambda_R-\eta K}{2a_0},\qquad
B_\eta=L+\sqrt{Q_\eta^2+\beta K^2}-Q_\eta.
\label{eq:Beta}
\end{equation}
Lemma~\ref{lem:moments} gives $0<\eta<1$, hence $Q_\eta>0$ on $0<s\le1$.

Let $\mu=\mu(n,K,R)$ be the first eigenvalue of the mixed model
\eqref{eq:kasue-model}, with the weighted realization in
Lemma~\ref{lem:endpoint} at $s=1$.

\begin{theorem}[Global inradius bound]
\label{thm:main}
Under \eqref{eq:assumptions},
\begin{equation}
\lambda_1^D(M)\ge\max\{nK,B_\eta\}.
\label{eq:main-bound}
\end{equation}
For $0<s<1$ one has the strict model estimate
\begin{equation}
B_\eta<\mu<\lambda_1^D(M).
\label{eq:strict-model}
\end{equation}
At $s=1$, the weighted Kasue model has first eigenvalue
$\mu=nK$ and $B_\eta<nK$.
\end{theorem}

\begin{corollary}[Small-$s$ accuracy]
\label{cor:asymptotics}
For fixed $n$, as $s\downarrow0$,
\begin{equation}
\frac{X}{\Lambda_R}
=1+cs+\beta a_0s^2+\beta a_1s^3+O_n(s^4),
\qquad X\in\{\mu,B_\eta\}.
\label{eq:main-expansion}
\end{equation}
Consequently,
\begin{equation}
\mu-B_\eta=O_n(\Lambda_Rs^4),\qquad
\frac{\mu-B_\eta}{\mu}=O_n(s^4).
\label{eq:error}
\end{equation}
\end{corollary}

Here $O_n$ allows the remainder constants to depend on $n$.

\begin{theorem}[Geometric sharpness at fixed data]
\label{thm:geometric-sharpness}
Fix $n\ge2$, $K>0$, and $0<R\le\pi/(2\sqrt K)$.
Let $\mathcal C(n,K,R)$ consist of the compact connected
$n$-dimensional Riemannian manifolds with smooth nonempty boundary
satisfying \eqref{eq:assumptions} and $\InRad(M)=R$. Then
\[
\inf_{M\in\mathcal C(n,K,R)}\lambda_1^D(M)=\mu(n,K,R).
\]
For $R<\pi/(2\sqrt K)$ the infimum is not attained, but is approached
by smooth rotationally symmetric balls with sectional curvature at
least $K$, totally geodesic boundary, and inradius exactly $R$.
At $R=\pi/(2\sqrt K)$ it is attained by the round hemisphere of
sectional curvature $K$.
\end{theorem}

The proof in Appendix~\ref{app:geometric-sharpness} combines
Kasue's comparison and equality classification
\cite[Theorem~2.1 and Definition~1.2]{Kasue} with a smooth cap construction.

\paragraph{Comparison and scope.}
Among explicit inradius bounds, Yang obtained
$\lambda_1^D(M)\ge\Lambda_R+\alpha/2$ \cite{Yang}, and Ling
sharpened this to $\lambda_1^D(M)\ge L$ \cite{Ling}. A variance
refinement gives a further improvement \cite{SchuermannVariance}.
The term $nK$ in \eqref{eq:main-bound} comes from Reilly's estimate
\cite{Reilly}:
\begin{equation}
\lambda_1^D(M)\ge nK.
\label{eq:reilly}
\end{equation}
The bound $B_\eta$ strictly improves the variance refinement
(Corollary~\ref{cor:variance-comparison}). Theorem~\ref{thm:crossover}
gives a direct algebraic comparison with the optimized power-test
benchmark. For each fixed $n$, $B_\eta$ exceeds that benchmark for
sufficiently small $s$; the combined estimate $\max\{nK,B_\eta\}$
does not dominate it in all dimensions.

The analytic construction combines the positive moment representation
in Lemma~\ref{lem:moments} with the sharp pointwise estimate in
Lemma~\ref{lem:corrector}. The Cauchy--Schwarz moment envelope fits
the classical relation between moments, quadrature, and Pad\'e
approximation \cite{AllenEtAl}.
For the spectral transfer used here, Corollary~\ref{cor:moment-transfer}
identifies $B_\eta$ as the optimum determined by the first two moments.
The geometric infimum is $\mu$
(Theorem~\ref{thm:geometric-sharpness}).

Section~\ref{sec:bound} proves the algebraic bound and model accuracy;
Section~\ref{sec:power} compares it with explicit benchmarks.
Appendix~\ref{app:geometric-sharpness} proves fixed-data geometric
sharpness.

\section{Two-moment construction and model accuracy}
\label{sec:bound}

We reduce the model to a Schr\"odinger operator, establish the moment
and corrector estimates, and combine them to prove the bound.

\subsection{Model and Schr\"odinger reduction}
For $0<s<1$, Kasue's eigenvalue comparison
\cite[Theorem~2.1]{Kasue} yields
\begin{equation}
\lambda_1^D(M)\ge\mu,
\label{eq:kasue-comparison}
\end{equation}
where $\mu$ is the first eigenvalue of the regular mixed problem
\begin{equation}
-\phi''+2c\sqrt K\tan(\sqrt K\,r)\phi'=\mu\phi,
\qquad \phi(0)=0,\quad \phi'(R)=0.
\label{eq:kasue-model}
\end{equation}
Kasue uses the inward unit normal; under our convention, the hypothesis
$H_{\partial M}\ge0$ corresponds to his boundary parameter $A=0$.
Savo's Robin comparison recovers this case in the Dirichlet limit
\cite[Theorem~1]{Savo}. At $s=1$ the model is singular; its weighted
realization is specified in Lemma~\ref{lem:endpoint}.

For $0<s<1$, write $k=\sqrt K$ and set
\begin{equation}
W(r)=K\tan^2(kr),\qquad
\psi_0(r)=\cos(\sqrt{\Lambda_R}\,r),
\label{eq:Wu0}
\end{equation}
and
\begin{equation}
J(s)=\frac4\pi\int_0^{\pi/2}
\tan^2(\sqrt{s}\,x)\cos^2x\,dx.
\label{eq:J}
\end{equation}
Here $\psi_0$ is the free Neumann--Dirichlet ground state. Define the
normalized average of $W$ with weight $\psi_0^2$ by
\[
\overline W:=
\frac{\int_0^R W(r)\psi_0(r)^2\,dr}
     {\int_0^R \psi_0(r)^2\,dr}
=KJ(s).
\]
The free ground-state eigenvalue $\Lambda_R$, together with the shift
$\alpha$, gives the baseline $L$.

\begin{lemma}[Schr\"odinger reduction]
\label{lem:reduction}
Let $\lambda_\beta^{ND}$ be the first eigenvalue of
\[
H_\beta=-\frac{d^2}{dr^2}+\beta W
\]
on $[0,R]$ with Neumann condition at $0$ and Dirichlet condition at
$R$. Then
\begin{equation}
\mu=\alpha+\lambda_\beta^{ND},
\qquad
\mu\le L+\beta KJ(s).
\label{eq:reduction}
\end{equation}
\end{lemma}

\begin{proof}
Let $\phi>0$ on $(0,R]$ be the first eigenfunction of
\eqref{eq:kasue-model}. With
$p(r)=\cos^{2c}(kr)$, the equation is
$-(p\phi')'=\mu p\phi$. Its mixed Rayleigh quotient gives $\mu>0$.
Integrating from $r$ to $R$ therefore yields
\[
p(r)\phi'(r)=\mu\int_r^R p(t)\phi(t)\,dt>0
\qquad(0\le r<R),
\]
so $y:=\phi'$ is positive on $[0,R)$. Put
$F=2ck\tan(kr)$. Differentiating gives
\begin{equation}
-y''+Fy'+F'y=\mu y,
\qquad y'(0)=0,\quad y(R)=0.
\label{eq:derivative}
\end{equation}
Here $y'(0)=\phi''(0)=0$ follows directly from
\eqref{eq:kasue-model} at $r=0$.

The gauge transform $v=\cos^c(kr)y$ yields
\[
-v''+c(c+1)K\tan^2(kr)v=(\mu-cK)v.
\]
The boundary conditions are $v'(0)=v(R)=0$, and $v>0$ on $(0,R)$.
Sturm oscillation identifies $v$ as the ground state, so
$\mu=\alpha+\lambda_\beta^{ND}$.

Testing the Rayleigh quotient of $H_\beta$ with
$\psi_0(r)=\cos(\sqrt{\Lambda_R}r)$ gives
$\lambda_\beta^{ND}\le\Lambda_R+\beta\overline W$, and therefore the
upper bound in \eqref{eq:reduction}.
\end{proof}

\subsection{Moment envelope and model representation}
We first derive a two-moment envelope and then represent the model
average in the required form.

For a probability measure $\nu$ on $[0,1]$, define
\begin{equation}
J_\nu(s)=s\int_{[0,1]}\frac{y}{1-sy}\,d\nu(y),
\qquad 0<s\le1.
\label{eq:moment-transform}
\end{equation}
At $s=1$ the integral is understood in the extended sense.

\begin{theorem}[Sharp two-moment envelope]
\label{thm:moment-optimality}
Let $0<m_1<1$ and $m_1^2<m_2<m_1$. Let $\mathcal P(m_1,m_2)$ be
the set of probability measures on $[0,1]$ with first and second
moments $m_1$ and $m_2$. Then
\begin{equation}
\min_{\nu\in\mathcal P(m_1,m_2)}J_\nu(s)
=\frac{m_1s}{1-(m_2/m_1)s}
\qquad(0<s\le1).
\label{eq:moment-envelope}
\end{equation}
For each $s$, the unique minimizing measure is
\begin{equation}
\nu_*=
\left(1-\frac{m_1^2}{m_2}\right)\delta_0
+\frac{m_1^2}{m_2}\delta_{m_2/m_1},
\label{eq:extremal-law}
\end{equation}
where $\delta_x$ is the Dirac measure at $x$.
\end{theorem}

\begin{proof}
Fix $\nu\in\mathcal P(m_1,m_2)$ and $0<s\le1$. If
$J_\nu(s)=+\infty$, the inequality is immediate. In the finite case,
$\nu(\{1\})=0$ when $s=1$. Applying Cauchy--Schwarz to the
factorization $y=\sqrt{y/(1-sy)}\sqrt{y(1-sy)}$ gives
\[
m_1^2
\le\left(\int\frac{y}{1-sy}\,d\nu\right)
   \left(\int y(1-sy)\,d\nu\right)
=\frac{J_\nu(s)}s(m_1-sm_2).
\]
Since $m_1-sm_2>0$, this gives the lower bound in
\eqref{eq:moment-envelope}.
Equality holds precisely when $y$ is constant on $\{y>0\}$,
$\nu$-almost surely. Its value must be $m_2/m_1$; the total mass
and first moment therefore determine \eqref{eq:extremal-law} uniquely.
Direct substitution shows that this law has the prescribed second
moment and attains the bound for every $s$.
\end{proof}

To apply this envelope to $J(s)$, we need a positive representation
with the moments in \eqref{eq:constants}.

\begin{lemma}[Positive moment representation]
\label{lem:moments}
There is a probability measure $\nu_0$ on $[0,1]$ with
$\nu_0((0,1))=1$ and nonconstant coordinate variable $Y$ such that
\begin{equation}
J(s)=J_{\nu_0}(s),\qquad
\E Y=a_0,\quad \E Y^2=a_1,
\qquad 0<s\le1.
\label{eq:model-moments}
\end{equation}
In particular, $0<a_0<1$ and $a_0^2<a_1<a_0$. Moreover,
$J(1)=1$ and
\begin{equation}
J(s)=a_0s+a_1s^2+O(s^3)
\qquad(s\downarrow0).
\label{eq:Jexp}
\end{equation}
\end{lemma}

\begin{proof}
Logarithmic differentiation of the cosine product
\cite[Eq.~4.22.2]{DLMF} gives
\[
\frac{\tan z}{z}=\frac8{\pi^2}\sum_{j=0}^{\infty}
\frac1{(2j+1)^2-4z^2/\pi^2},\qquad |z|<\frac\pi2.
\]
On $[0,\pi/2]\times\mathbb N_0$ define
\[
d\mathbb P(x,j)=\frac{32x\sin(2x)}{\pi^3(2j+1)^2}\,dx,
\qquad
Y(x,j)=\frac{4x^2}{\pi^2(2j+1)^2}.
\]
The choice of $Y$ matches the partial-fraction denominators;
$\sin(2x)$ comes from integrating the weight $\cos^2x$ by parts.
Since $\sum_{j\ge0}(2j+1)^{-2}=\pi^2/8$ and
$\int_0^{\pi/2}x\sin(2x)\,dx=\pi/4$, this is a probability measure.
The law $\nu_0$ of $Y$ assigns full mass to $(0,1)$, and $Y$ is
nonconstant.
Tonelli's theorem and the partial-fraction identity give, for $0<s<1$,
\begin{align*}
\E\frac1{1-sY}
&=\frac{32}{\pi^3}\int_0^{\pi/2}x\sin(2x)
  \sum_{j=0}^{\infty}
  \frac{dx}{(2j+1)^2-4sx^2/\pi^2}\\
&=\frac4{\pi\sqrt s}\int_0^{\pi/2}
  \sin(2x)\tan(\sqrt s\,x)\,dx\\
&=\frac4\pi\int_0^{\pi/2}
  \cos^2x\,\sec^2(\sqrt s\,x)\,dx
=1+J(s).
\end{align*}
The integration by parts uses
$\sin(2x)=-(\cos^2x)'$ and vanishing boundary terms. Subtracting $1$
gives $J(s)=s\E[Y/(1-sY)]$. At $s=1$, direct integration gives
\[
J(1)=\frac4\pi\int_0^{\pi/2}\sin^2x\,dx=1,
\]
and monotone convergence extends the representation to this endpoint.

Expanding $(1-sY)^{-1}$ uniformly near $s=0$ and using the Taylor
series of $\tan^2(\sqrt s\,x)$ yields
\begin{align*}
\E Y
&=\frac4\pi\int_0^{\pi/2}x^2\cos^2x\,dx
=\frac{\pi^2-6}{12}=a_0,\\
\E Y^2
&=\frac8{3\pi}\int_0^{\pi/2}x^4\cos^2x\,dx
=\frac{\pi^4-20\pi^2+120}{120}=a_1.
\end{align*}
The strict moment inequalities follow from $0<Y<1$ almost surely
and the nonconstancy of $Y$. The same expansion proves \eqref{eq:Jexp}.
\end{proof}

\begin{lemma}[Rational moment minorant]
\label{lem:moment-minorant}
One has $0<\eta<1$, $J(1)=1$, and
\begin{equation}
J(s)>j_\eta(s)=\frac{a_0s}{1-\eta s}>0,
\qquad 0<s\le1.
\label{eq:moment-minorant}
\end{equation}
\end{lemma}

\begin{proof}
Lemma~\ref{lem:moments} gives the strict moment inequalities and
$J(1)=1$, hence also $0<\eta<1$. Its model law $\nu_0$ gives no mass
to $0$ and therefore differs from the unique extremal law
\eqref{eq:extremal-law}. Theorem~\ref{thm:moment-optimality} gives
\eqref{eq:moment-minorant}. The minorant agrees with \eqref{eq:Jexp}
through $s^2$.
\end{proof}

\subsection{Sharp corrector estimate}
The corrector below modifies the free ground state into a comparison
function. Its pointwise bound controls the quadratic term produced
by this modification.

\begin{lemma}[Sharp corrector estimate]
\label{lem:corrector}
For $0<s<1$, define on $[0,R)$
\begin{equation}
f(r)=\psi_0(r)^{-2}
\int_0^r(\overline W-W(t))\psi_0(t)^2\,dt.
\label{eq:f}
\end{equation}
Then $f$ extends smoothly to $[0,R]$, vanishes at both endpoints, and
satisfies the sharp estimates
\begin{align}
0<f(r)&<\frac{\overline W}{k}\tan(kr)
\qquad(0<r<R),\label{eq:f-bound}\\
\sup_{0<r<R}\frac{f(r)^2}{W(r)}&=J(s)^2.
\label{eq:sharp-corrector}
\end{align}
\end{lemma}

\begin{proof}
Let $A(r)$ denote the integral in \eqref{eq:f}. Since $W$ is strictly
increasing, $\overline W-W$ changes sign once, from positive to negative.
Thus $A$ first increases and then decreases. Its endpoint values are
zero by the definition of $\overline W$, so $A>0$ on $(0,R)$.

Near $R$, the factor $\psi_0^2$ has a zero of order two, whereas
$\overline W-W$ is smooth for $s<1$. Consequently
$A(r)=-\int_r^R(\overline W-W)\psi_0^2\,dt$ has a zero of order three:
\[
A(r)=(R-r)^3a_R(r),\qquad
\psi_0(r)=(R-r)b_R(r),
\]
with smooth $a_R,b_R$ and
$b_R(R)=\sqrt{\Lambda_R}\ne0$. Hence
$f=(R-r)a_R/b_R^2$ extends smoothly and vanishes at $R$; smoothness and
$f(0)=0$ at the left endpoint are immediate from \eqref{eq:f}.

To prove the upper bound, put $a=\sqrt{\Lambda_R}>k$ and consider
the weighted difference
\[
Z=\psi_0^2\left[\frac{\overline W}{k}\tan(kr)-f\right].
\]
A differentiation gives
\[
Z'=\psi_0^2\tan^2(kr)
\left[\overline W+K
-2\overline W\frac ak\frac{\tan(ar)}{\tan(kr)}\right].
\]
The ratio $\tan(ar)/\tan(kr)$ is strictly increasing. Indeed,
\[
\frac{d}{dr}\log\frac{\tan(ar)}{\tan(kr)}
=2\left(\frac a{\sin(2ar)}-\frac k{\sin(2kr)}\right)>0,
\]
because $0<2kr<2ar<\pi$ and $t/\sin t$ is strictly increasing on
$(0,\pi)$. Hence the bracket in the formula for $Z'$ is strictly
decreasing.

Since $Z(0)=Z(R)=0$, one has $\int_0^R Z'\,dr=0$.
The prefactor $\psi_0^2\tan^2(kr)$ is positive on $(0,R)$, so the
strictly decreasing bracket takes both signs and vanishes exactly once.
Consequently $Z$ increases from zero and then decreases to zero,
remaining positive on $(0,R)$. This proves the upper bound in
\eqref{eq:f-bound}; the lower bound follows from $A>0$.
Since $\sqrt W=k\tan(kr)$ and $\overline W=KJ(s)$, it follows that
$f^2<J(s)^2W$ in $(0,R)$. Finally, \eqref{eq:f} gives
$f'(0)=\overline W$, whereas $W(r)/r^2\to K^2$ as $r\downarrow0$.
Therefore $f(r)^2/W(r)\to J(s)^2$, proving \eqref{eq:sharp-corrector}.
\end{proof}

\subsection{Spectral transfer and proof of the bound}
To convert a lower bound for $J$ into an eigenvalue bound, define
\begin{equation}
\Phi(j)=\frac{K}{2j}\bigl(\sqrt{1+4\beta j^2}-1\bigr)
=\frac{2\beta Kj}{1+\sqrt{1+4\beta j^2}},\qquad j>0,
\label{eq:Phi}
\end{equation}
with its continuous extension $\Phi(+\infty)=K\sqrt\beta$.
Writing $\Phi(j)=\sqrt{q^2+\beta K^2}-q$ with $q=K/(2j)$ shows
that $\Phi$ is strictly increasing.

\begin{proof}[Proof of Theorem~\ref{thm:main} in the regular range]
Assume $0<s<1$. Choose $\tau>0$ and incorporate the corrector
into a positive comparison function $h$ by
\begin{align}
\tau+J(s)^2\tau^2&=\beta,\label{eq:tau}\\
h(r)&=\psi_0(r)\exp\!\left(-\tau\int_0^r f(t)\,dt\right).
\label{eq:comparison-function}
\end{align}
By Lemma~\ref{lem:corrector}, $h$ is smooth on $[0,R]$, positive on
$[0,R)$, and satisfies $h'(0)=h(R)=0$. The same lemma gives
$f^2<J(s)^2W$ on $(0,R)$. Direct differentiation using $-\psi_0''=\Lambda_R\psi_0$,
$(\psi_0^2f)'=(\overline W-W)\psi_0^2$, and \eqref{eq:tau} gives
\[
(H_\beta-\Lambda_R-\tau\overline W)h
=\tau^2\bigl[J(s)^2W-f^2\bigr]h>0
\quad(0<r<R).
\]
Let $v$ be the positive first eigenfunction of $H_\beta$. Green's
identity and the common mixed boundary conditions yield
\begin{equation}
(\lambda_\beta^{ND}-\Lambda_R-\tau\overline W)
\int_0^R hv\,dr
=\tau^2\int_0^R\bigl[J(s)^2W-f^2\bigr]hv\,dr>0.
\label{eq:comparison-identity}
\end{equation}
Solving \eqref{eq:tau} gives $\tau\overline W=\Phi(J(s))$, so
Lemma~\ref{lem:reduction} implies
\begin{equation}
\mu>L+\Phi(J(s)).
\label{eq:model-lower}
\end{equation}
By \eqref{eq:moment-minorant} and monotonicity of $\Phi$,
$L+\Phi(J(s))>L+\Phi(j_\eta(s))=B_\eta$, since
$K/(2j_\eta(s))=Q_\eta$. Kasue's comparison and Reilly's estimate
therefore prove \eqref{eq:main-bound} for $s<1$.
Under our assumptions, Kasue's equality classification
\cite[Theorem~2.1 and Definition~1.2]{Kasue} permits equality only for
the round hemisphere, which has $s=1$. Hence
$\mu<\lambda_1^D(M)$ for $s<1$, completing \eqref{eq:strict-model}.
\end{proof}

For $0<s<1$ and $\tau\ge0$, the sharp constant also gives
\[
(\beta-\tau)W-\tau^2f^2\ge0\ \text{on }(0,R)
\quad\Longleftrightarrow\quad
\tau+J(s)^2\tau^2\le\beta.
\]
Thus \eqref{eq:tau} selects the largest admissible coefficient in the
comparison family \eqref{eq:comparison-function}.

\subsection{The singular endpoint \texorpdfstring{$s=1$}{s=1}}
At $s=1$, the model weight vanishes at $R$. Its weighted realization
completes the proof of Theorem~\ref{thm:main}.

\begin{lemma}[The endpoint $s=1$]
\label{lem:endpoint}
At $s=1$, the weighted Kasue model has first eigenvalue $nK$.
Moreover $B_\eta<nK$.
\end{lemma}

\begin{proof}
At $s=1$, let $R=\pi/(2\sqrt K)$ and
$p(r)=\cos^{n-1}(\sqrt K\,r)$. The divergence-form operator
$-p^{-1}(p\phi')'$ acts in $L^2((0,R),p(r)\,dr)$ with form domain
\begin{align*}
\mathcal V_p
&=\overline{\{\psi\in C^\infty([0,R]):\psi(0)=0\}}^{\,\|\cdot\|_{H^1_p}},\\
\|\psi\|_{H^1_p}^2
&=\int_0^R p(r)\bigl(|\psi'(r)|^2+|\psi(r)|^2\bigr)\,dr.
\end{align*}
Since $p$ is bounded away from zero near $0$, the trace there is
continuous, and every $\psi\in\mathcal V_p$ satisfies $\psi(0)=0$.
The model operator is the self-adjoint operator associated with the
closed quadratic form
$\mathfrak a_p[\psi]=\int_0^R p(r)|\psi'(r)|^2\,dr$ on
$\mathcal V_p$. Its first eigenvalue is characterized by
\[
\mu=\inf_{\psi\in\mathcal V_p\setminus\{0\}}\mathcal Q[\psi],
\qquad
\mathcal Q[\psi]
=\frac{\int_0^R p(r)|\psi'(r)|^2\,dr}
       {\int_0^R p(r)|\psi(r)|^2\,dr}.
\]
The form domain prescribes no boundary condition at $R$; the condition
there is natural for the degenerate weight. The positive function
$\phi(r)=\sin(\sqrt K\,r)$ satisfies
$-(p\phi')'=nKp\phi$. For a smooth test function
$\psi=\phi g$ with $\psi(0)=0$, the ground-state identity gives
\[
\int_0^R p\bigl(|\psi'|^2-nK|\psi|^2\bigr)\,dr
=\int_0^R p\phi^2|g'|^2\,dr\ge0;
\]
the boundary term vanishes because $\phi(0)=0$ and
$p(R)\phi'(R)=0$. By density, the same inequality holds on the form
domain. Since $\phi\in\mathcal V_p$ and $\mathcal Q[\phi]=nK$, the
minimum is attained at $\phi$ and the first weighted eigenvalue is
$\mu=nK$.

For the algebraic bound, $\Lambda_R=K$ and
$j_\eta(1)<J(1)=1$ by Lemma~\ref{lem:moment-minorant}. Hence
\[
B_\eta=L+\Phi(j_\eta(1))<L+\Phi(1)=nK,
\]
since $\sqrt{1+4\beta}=2c+1=n$.
Together with Reilly's estimate \eqref{eq:reilly}, this completes
Theorem~\ref{thm:main} at $s=1$.
\end{proof}

\subsection{Model accuracy and two-moment optimality}
\begin{proof}[Proof of Corollary~\ref{cor:asymptotics}]
We obtain the model expansion by comparing our lower bound with the
Rayleigh upper bound. Lemma~\ref{lem:moments} and $\eta=a_1/a_0$ give
\[
j_\eta(s)=a_0s+a_1s^2+O(s^3),
\qquad
J(s)=a_0s+a_1s^2+O(s^3).
\]
For $j\downarrow0$,
\[
\Phi(j)=\beta Kj+O_n(Kj^3).
\]
Since $K=\Lambda_Rs$,
\[
\frac{B_\eta}{\Lambda_R}
=1+cs+\beta a_0s^2+\beta a_1s^3+O_n(s^4).
\]
Lemma~\ref{lem:reduction} and \eqref{eq:model-lower} give
\[
B_\eta<\mu\le L+\beta KJ(s),
\]
and the upper expression has the same expansion through $s^3$. This
proves \eqref{eq:main-expansion}. Since $\mu/\Lambda_R\to1$,
\eqref{eq:error} follows.
\end{proof}

\begin{samepage}
The moment envelope also determines the optimum for the transfer $\Phi$.

\begin{corollary}[Optimal two-moment transfer]
\label{cor:moment-transfer}
For $n\ge2$, $K>0$, and $0<s\le1$,
\begin{equation}
\min_{\nu\in\mathcal P(a_0,a_1)}
\bigl[L+\Phi(J_\nu(s))\bigr]
=L+\Phi(j_\eta(s))=B_\eta.
\label{eq:moment-transfer}
\end{equation}
Consequently, if a positive function $\underline J(s)$ satisfies
$\underline J(s)\le J_\nu(s)$ for every
$\nu\in\mathcal P(a_0,a_1)$, then
$L+\Phi(\underline J(s))\le B_\eta$.
\end{corollary}
\end{samepage}

\begin{proof}[Proof of Corollary~\ref{cor:moment-transfer}]
Strict monotonicity of $\Phi$ and
Theorem~\ref{thm:moment-optimality} give the minimum in
\eqref{eq:moment-transfer}. Since $K/(2j_\eta(s))=Q_\eta$,
\eqref{eq:Phi} identifies its value with $B_\eta$.
Any universally valid moment minorant $\underline J$ is at most
$j_\eta$ by \eqref{eq:moment-envelope}; monotonicity proves the final
assertion.
\end{proof}

\section{Comparison with explicit benchmarks}
\label{sec:power}

We compare $B_\eta$ with the variance refinement and the optimized
power-test bound. The latter improves \eqref{eq:main-bound} only
where it exceeds both $B_\eta$ and $nK$.

\subsection{Variance refinement}
The variance refinement \cite[Theorem~1.1]{SchuermannVariance} is
\begin{equation}
V=4\zeta(3)-\frac{\pi^2+4}{3},\qquad
B_{\mathrm{var}}=\frac{L+\sqrt{L^2+V\alpha^2}}{2},
\label{eq:Bvar}
\end{equation}
where $\zeta$ is the Riemann zeta function.

\begin{corollary}[Comparison with the variance refinement]
\label{cor:variance-comparison}
For every $n\ge2$ and $0<s\le1$,
\begin{equation}
B_\eta>B_{\mathrm{var}}.
\label{eq:variance-dominance}
\end{equation}
\end{corollary}

\begin{proof}
The elementary estimates
\[
1+\frac18+\frac1{27}+\frac1{64}<\zeta(3)<\frac54,
\qquad 9<\pi^2<10,
\]
give $0<V<2/3<1<4a_0$ and $0<a_0<1$. The upper bound for $\zeta(3)$
follows by comparing $\sum_{j\ge3}j^{-3}$ with
$\int_2^\infty t^{-3}\,dt=1/8$.
Compare the two corrections through their quadratic equations.
Put $x=B_{\mathrm{var}}-L>0$. Equation~\eqref{eq:Bvar} gives
\[
x^2+Lx=\frac V4\alpha^2.
\]
Since $a_0<1$ and $2Q_\eta\le L/a_0$, one has
\[
x^2+2Q_\eta x
\le\frac1{a_0}(x^2+Lx)
=\frac{V}{4a_0}\alpha^2<\alpha^2<\beta K^2.
\]
By \eqref{eq:Beta}, $B_\eta-L$ is the positive root of
$z^2+2Q_\eta z=\beta K^2$. The left-hand side is strictly increasing
for $z\ge0$, so $x<B_\eta-L$, proving \eqref{eq:variance-dominance}.
\end{proof}

\subsection{Power-test benchmark}
The bound below is the mixed-model analogue of Shi--Zhang's diameter
power tests for closed manifolds and for Neumann problems with convex
boundary \cite{ShiZhang}. The optimization also appears for the
Witten Laplacian \cite{FutakiLiLi}. We derive the inradius version
from Chen's mixed-boundary identity \cite[Lemma~2.1]{Chen}.

\begin{proposition}[Power-test benchmark]
\label{prop:power}
Under \eqref{eq:assumptions},
\begin{equation}
\lambda_1^D(M)\ge B_{\mathrm{pow}},
\label{eq:power-bound}
\end{equation}
where
\begin{equation}
B_{\mathrm{pow}}=
\begin{cases}
L+\alpha^2/(4\Lambda_R),&\alpha\le2\Lambda_R,\\
2\alpha,&\alpha\ge2\Lambda_R.
\end{cases}
\label{eq:Bpow}
\end{equation}
\end{proposition}

\begin{proof}
For $0<s<1$, use $y=\phi'$ and
$F=2c\sqrt K\tan(\sqrt K r)$ in \eqref{eq:derivative}. Then
$y>0$ on $[0,R)$ and
$F'=2cK\sec^2(\sqrt K r)\ge2\alpha$. To balance the gradient and
curvature contributions, fix $p>1$ and put $z=y^{p/2}$. Since
$y'(R)=-\mu\phi(R)<0$, the zero of $y$ at $R$ is simple; hence
$z'(r)=O((R-r)^{p/2-1})$ and $z\in H^1(0,R)$ for every $p>1$.

Multiply \eqref{eq:derivative} by $y^{p-1}$ and integrate by parts on
$[0,R-\varepsilon]$. The terms at $0$ vanish because $y'(0)=F(0)=0$.
Since $y(r)=O(R-r)$, the boundary terms at $R-\varepsilon$ are
$O(\varepsilon^{p-1})$ and $O(\varepsilon^p)$; all integrands are
integrable for $p>1$. Letting $\varepsilon\downarrow0$ therefore gives
\[
\mu\int_0^Rz^2\,dr
=\frac{4(p-1)}{p^2}\int_0^R(z')^2\,dr
+\frac{p-1}{p}\int_0^R F'z^2\,dr.
\]
The mixed Poincar\'e inequality
$\int_0^R(z')^2\,dr\ge\Lambda_R\int_0^Rz^2\,dr$ gives
\begin{equation}
\mu\ge4t(1-t)\Lambda_R+2\alpha t,
\qquad t=\frac{p-1}{p}\in(0,1).
\label{eq:power-family}
\end{equation}
If $\alpha<2\Lambda_R$, the maximum occurs at
$t_*=\tfrac12+\alpha/(4\Lambda_R)$ and equals $L+\alpha^2/(4\Lambda_R)$.
For $\alpha\ge2\Lambda_R$, the supremum is $2\alpha$, approached as
$t\uparrow1$; the branches agree at $\alpha=2\Lambda_R$.
Thus \eqref{eq:power-family} gives $\mu\ge B_{\mathrm{pow}}$, and
Kasue's comparison proves \eqref{eq:power-bound} for $s<1$.
At $s=1$, $B_{\mathrm{pow}}\le nK$, so Reilly's estimate covers the endpoint.
\end{proof}

For $cs\le2$, the first branch gives the exact identity
\[
\frac{B_{\mathrm{pow}}}{\Lambda_R}=1+cs+\frac{c^2}{4}s^2.
\]
Since $a_0>1/4$ and $\beta>c^2$, Corollary~\ref{cor:asymptotics} implies
$B_\eta>B_{\mathrm{pow}}$ for sufficiently small $s$ in each fixed
dimension.

\subsection{Algebraic comparison with the power-test bound}
The next criterion compares the bounds at fixed $n$ and $s$, and
determines when the power-test benchmark also exceeds Reilly's estimate.

\begin{theorem}[Algebraic comparison]
\label{thm:crossover}
Let $u=cs$. For $0<u\le2$, define
\begin{equation}
D_c(u)=1-\frac1{4a_0}-\frac{u^2}{16}
+\frac{1}{c}+\frac{\eta u}{4a_0c}.
\label{eq:Dcu}
\end{equation}
Then
\begin{equation}
B_\eta\gtrless B_{\mathrm{pow}}
\quad\Longleftrightarrow\quad
D_c(u)\gtrless0.
\label{eq:exact-crossover}
\end{equation}
Equality holds if and only if $D_c(u)=0$. In the same range,
\begin{equation}
\frac{B_{\mathrm{pow}}-nK}{\Lambda_R}
=\left(1-\frac u2\right)^2-\frac uc.
\label{eq:power-reilly-comparison}
\end{equation}
If $u\ge2$, then $B_{\mathrm{pow}}=(n-1)K<nK$.
\end{theorem}

\begin{proof}
Assume $0<u=cs\le2$ and normalize by $\Lambda_R$. Here
$u=\alpha/\Lambda_R$ measures the curvature contribution relative
to the free interval eigenvalue. Put
\[
q=\frac{1-\eta s}{2a_0}>0,\qquad
 d=\frac{B_\eta-L}{\Lambda_R},\qquad
 w=\frac{B_{\mathrm{pow}}-L}{\Lambda_R}=\frac{u^2}{4}.
\]
By \eqref{eq:Beta},
\begin{equation}
d^2+2qd=\beta s^2.
\label{eq:dquad}
\end{equation}
Substituting the power-test correction $w$ into the quadratic
defining $d$ gives
\begin{equation}
w^2+2qw-\beta s^2=-u^2D_c(u).
\label{eq:criterion-identity}
\end{equation}
Since $x\mapsto x^2+2qx$ is strictly increasing on $[0,\infty)$,
\eqref{eq:dquad}--\eqref{eq:criterion-identity} imply
\eqref{eq:exact-crossover}. Since $nK/\Lambda_R=2u+u/c$, the
first branch of \eqref{eq:Bpow} also gives
\eqref{eq:power-reilly-comparison}. If $u\ge2$, the second branch of
\eqref{eq:Bpow} gives
$B_{\mathrm{pow}}=(n-1)K<nK$.
\end{proof}

Consequently, $B_{\mathrm{pow}}>\max\{nK,B_\eta\}$ holds precisely
when
\[
0<u<2,\qquad D_c(u)<0,\qquad
\left(1-\frac u2\right)^2>\frac uc.
\]
These conditions can hold in sufficiently high dimensions. Indeed,
put $A=1-1/(4a_0)$. Since $9<\pi^2<10$, one has $0<A<1/4$.
Fix $u\in(4\sqrt A,2)$ and let $c\to\infty$ with $s=u/c$. Then
\[
D_c(u)\longrightarrow A-\frac{u^2}{16}<0,
\qquad
\left(1-\frac u2\right)^2-\frac uc
\longrightarrow\left(1-\frac u2\right)^2>0.
\]
For all sufficiently large $n$, both strict inequalities hold at
$s=u/c\in(0,1)$, and continuity preserves them on a nonempty
interval of admissible $s$.

\paragraph{Dimensional scale.}
Numerical evaluation of the criterion places the onset of
$B_{\mathrm{pow}}>\max\{nK,B_\eta\}$ at dimensions of order $10^3$.
Crossings with $B_\eta$ alone may occur much earlier, but need not
improve Reilly's bound.

\paragraph{Further refinements.}
Temple's inequality \cite[Section~6]{SimonTemple} uses a trial Rayleigh
quotient, its residual variance, and a lower bound for the second
eigenvalue exceeding that quotient. Temple estimates can also have
algebraic form and may improve $B_\eta$; higher moments offer a further
route to refinement.

\appendix
\numberwithin{equation}{section}

\section{Geometric sharpness at fixed data}
\label{app:geometric-sharpness}

Related cap constructions occur in Milman's isoperimetric model
spaces \cite[Section~5]{Milman}; the model realizations of
Savo \cite[Section~2.2 and Appendix]{Savo} and
Li--Wang \cite[Remark~1.7]{LiWang} use mixed boundary conditions.
The construction below preserves the prescribed inradius and imposes
Dirichlet conditions on the entire boundary.

\begin{proof}[Proof of Theorem~\ref{thm:geometric-sharpness}]
Assume first $0<R<\pi/(2\sqrt K)$ and put $k=\sqrt K$.
We construct a smooth rotationally symmetric ball whose radial volume
density agrees with the model density outside a shrinking cap.

\paragraph{The cap.}
Fix $0<\delta<R$ and set
\[
h(t)=\cos(k(R-t)),\qquad 0\le t\le R.
\]
Then $h>0$, $h'>0$ on $[0,R)$, $h'(R)=0$, and $h''+Kh=0$.
Choose a nonzero $\rho_\delta\in C_c^\infty((0,\delta))$ with
$\rho_\delta\ge0$, and let
\[
a_\delta=\frac{h(0)}{\int_0^\delta h(s)\rho_\delta(s)\,ds},
\qquad
F_\delta''+KF_\delta=-a_\delta\rho_\delta,
\qquad F_\delta(0)=0,\quad F_\delta'(0)=1.
\]
Here $\rho_\delta$ is extended by zero outside $(0,\delta)$.
Near $0$, one has $F_\delta(t)=\sin(kt)/k$, so the metric below
extends smoothly across its pole.

For $\mathcal W=hF_\delta'-h'F_\delta$, differentiation gives
$\mathcal W'=-a_\delta h\rho_\delta$ and $\mathcal W(0)=h(0)$.
Thus
\[
\mathcal W(t)=a_\delta\int_t^\delta h(s)\rho_\delta(s)\,ds\ge0
\quad(0\le t\le\delta),\qquad
\mathcal W(t)=0\quad(\delta\le t\le R).
\]
Writing $y=F_\delta/h$, we have $y'=\mathcal W/h^2\ge0$.
Since $y(0)=0$ and $y'(0)>0$, it follows that $F_\delta>0$ on
$(0,R]$ and $F_\delta'=h'y+hy'>0$ on $(0,R)$.
Moreover $y$ is a positive constant $c_\delta$ on $[\delta,R]$,
so
\[
F_\delta(t)=c_\delta h(t)\quad(\delta\le t\le R),
\qquad F_\delta'(R)=0,
\qquad -F_\delta''\ge KF_\delta.
\]

\paragraph{Geometry.}
On the $n$-ball with pole $t=0$, use the metric
\[
g_\delta=dt^2+F_\delta(t)^2g_{\mathbb S^{n-1}},
\qquad 0\le t\le R,
\]
where $g_{\mathbb S^{n-1}}$ is the unit round metric.
The radial sectional curvature is $-F_\delta''/F_\delta\ge K$.
For $n\ge3$, the tangential sectional curvature is
$(1-(F_\delta')^2)/F_\delta^2$. To estimate it, observe that
\[
\frac{d}{dt}\bigl((F_\delta')^2+KF_\delta^2\bigr)
=2F_\delta'(F_\delta''+KF_\delta)\le0.
\]
The value at $0$ is $1$, hence
$1-(F_\delta')^2\ge KF_\delta^2$.
The warped-product curvature formula expresses the sectional curvature
of an arbitrary plane as a convex combination of the radial and
tangential values. Thus every sectional curvature is at least $K$;
for $n=2$ the radial calculation alone suffices.
The estimate extends across the pole by smoothness and implies
$\Ric_{g_\delta}\ge(n-1)Kg_\delta$.

The outward normal at $t=R$ is $\partial_t$, and the second
fundamental form is
$(F_\delta'(R)/F_\delta(R))g_{\partial M_\delta}=0$.
Furthermore, any curve from radial level $t$ to the boundary has
length at least $R-t$, since its speed is at least the absolute
radial speed. A radial segment realizes this length. Therefore
\[
\operatorname{dist}(x,\partial M_\delta)=R-t(x),
\qquad \InRad(M_\delta)=R.
\]
Each $M_\delta$ consequently belongs to $\mathcal C(n,K,R)$.

\paragraph{The Rayleigh limit.}
Let $\phi$ be the positive first eigenfunction of
\eqref{eq:kasue-model}, and set $u(t)=\phi(R-t)$.
This defines a Lipschitz function on $M_\delta$ with zero boundary
trace, hence an element of $H_0^1(M_\delta)$.
The Rayleigh principle gives
\[
\lambda_1^D(M_\delta)
\le
\frac{\int_0^R |u'(t)|^2F_\delta(t)^{n-1}\,dt}
     {\int_0^R |u(t)|^2F_\delta(t)^{n-1}\,dt}.
\]
The common spherical volume factor cancels. Define
\[
\omega_\delta(t)=\left(\frac{F_\delta(t)}{c_\delta}\right)^{n-1},
\qquad \omega(t)=h(t)^{n-1}.
\]
The monotonicity of $F_\delta/h$ yields
$0\le\omega_\delta\le\omega\le1$, with equality
$\omega_\delta=\omega$ on $[\delta,R]$. In particular,
\[
\|\omega_\delta-\omega\|_{L^1(0,R)}\le\delta\longrightarrow0.
\]
Divide the Rayleigh numerator and denominator by
$c_\delta^{n-1}$. Since $u$ and $u'$ are bounded, the resulting
quotient converges to
\[
\frac{\int_0^R |u'|^2\omega\,dt}{\int_0^R |u|^2\omega\,dt}
=\frac{\int_0^R |\phi'(r)|^2\cos^{n-1}(kr)\,dr}
       {\int_0^R |\phi(r)|^2\cos^{n-1}(kr)\,dr}
=\mu(n,K,R).
\]
The limiting denominator is positive. Kasue's lower bound
\eqref{eq:kasue-comparison} now gives
\[
\mu(n,K,R)\le\lambda_1^D(M_\delta)\longrightarrow\mu(n,K,R),
\]
which proves the infimum identity in the regular range.
The balls collapse: $F_\delta''<0$ and $F_\delta'(0)=1$ imply
$F_\delta(\delta)\le\delta$, so
$c_\delta\le\delta/h(\delta)\to0$ and $F_\delta\le c_\delta$.

Finally, Kasue's equality classification
\cite[Theorem~2.1 and Definition~1.2]{Kasue} permits equality for
$K>0$ and boundary parameter $A=0$ only for the round hemisphere,
whose inradius is $\pi/(2\sqrt K)$. Thus the infimum is not attained
in the regular range. At $R=\pi/(2\sqrt K)$, the hemisphere has
first Dirichlet eigenvalue $nK$, equal to the weighted model eigenvalue
in Lemma~\ref{lem:endpoint}; the comparison bound is therefore attained.
\end{proof}

\section*{Statements and Declarations}

\smallskip
\noindent\textbf{Funding and competing interests.}
No funding or other support was received for this work. The author
declares no competing interests.

\smallskip
\noindent\textbf{Data availability.}
All results are analytic. No datasets are associated with this work.

\smallskip
\noindent\textbf{Use of generative artificial intelligence.}
Generative AI assisted with language editing and literature searches. The author assumes full responsibility for the mathematical content, references, and final manuscript.


\clearpage
\begin{thebibliography}{99}
\raggedright

\bibitem{Kasue}
Kasue, A.: On a lower bound for the first eigenvalue of the Laplace
operator on a Riemannian manifold. Ann. Sci. \'Ecole Norm. Sup. (4)
\textbf{17}(1), 31--44 (1984).
\url{https://doi.org/10.24033/asens.1464}

\bibitem{Yang}
Yang, D.: Lower bound estimates of the first eigenvalue for compact
manifolds with positive Ricci curvature. Pacific J. Math. \textbf{190}(2),
383--398 (1999).
\url{https://doi.org/10.2140/pjm.1999.190.383}

\bibitem{Ling}
Ling, J.: A lower bound of the first Dirichlet eigenvalue of a compact
manifold with positive Ricci curvature. Int. J. Math. \textbf{17}(5),
605--617 (2006).
\url{https://doi.org/10.1142/S0129167X06003631}

\bibitem{SchuermannVariance}
Sch\"urmann, T.: Variance-refined in-diameter lower bound for the first
Dirichlet eigenvalue. Int. J. Math., article 2650086, 12 pp. (2026).
\url{https://doi.org/10.1142/S0129167X26500862}

\bibitem{Reilly}
Reilly, R.C.: Applications of the Hessian operator in a Riemannian
manifold. Indiana Univ. Math. J. \textbf{26}(3), 459--472 (1977).
\url{https://doi.org/10.1512/iumj.1977.26.26036}

\bibitem{AllenEtAl}
Allen, G.D., Chui, C.K., Madych, W.R., Narcowich, F.J., Smith, P.W.:
Pad\'e approximation and Gaussian quadrature.
Bull. Austral. Math. Soc. \textbf{11}(1), 63--69 (1974).
\url{https://doi.org/10.1017/S0004972700043641}

\bibitem{Savo}
Savo, A.: Optimal eigenvalue estimates for the Robin Laplacian on
Riemannian manifolds. J. Differential Equations \textbf{268}(5),
2280--2308 (2020).
\url{https://doi.org/10.1016/j.jde.2019.09.013}

\bibitem{LiWang}
Li, X., Wang, K.: First Robin eigenvalue of the $p$-Laplacian on
Riemannian manifolds. Math. Z. \textbf{298}(3--4), 1033--1047 (2021).
\url{https://doi.org/10.1007/s00209-020-02645-y}

\bibitem{Milman}
Milman, E.: Sharp isoperimetric inequalities and model spaces for the
Curvature-Dimension-Diameter condition. J. Eur. Math. Soc.
\textbf{17}(5), 1041--1078 (2015).
\url{https://doi.org/10.4171/JEMS/526}

\bibitem{DLMF}
NIST Digital Library of Mathematical Functions, Eq.~4.22.2.
\url{https://dlmf.nist.gov/4.22.E2}

\bibitem{ShiZhang}
Shi, Y.-M., Zhang, H.-C.: Lower bounds for the first eigenvalue on
compact manifolds. Chinese Ann. Math. Ser. A \textbf{28}(6), 863--866
(2007). (Chinese, with English summary.)

\bibitem{FutakiLiLi}
Futaki, A., Li, H., Li, X.-D.: On the first eigenvalue of the
Witten--Laplacian and the diameter of compact shrinking solitons.
Ann. Global Anal. Geom. \textbf{44}, 105--114 (2013).
\url{https://doi.org/10.1007/s10455-012-9358-5}

\bibitem{Chen}
Chen, M.-F.: General estimate of the first eigenvalue on manifolds.
Front. Math. China \textbf{6}(6), 1025--1043 (2011).
\url{https://doi.org/10.1007/s11464-011-0164-3}

\bibitem{SimonTemple}
Simon, B.: Tosio Kato's work on non-relativistic quantum mechanics:
part 1. Bull. Math. Sci. \textbf{8}, 121--232 (2018).
\url{https://doi.org/10.1007/s13373-018-0118-0}

\end{thebibliography}
\end{document}